\documentclass{amsart}

\usepackage[
left=1.25in, right=1.25in]{geometry}
\usepackage[all]{xy}
\usepackage{graphicx}
\usepackage{indentfirst}
\usepackage{bm}
\usepackage{amsthm}
\usepackage{mathrsfs}
\usepackage{latexsym}
\usepackage{amsmath}
\usepackage{amssymb}
\usepackage{hyperref}
\usepackage{microtype}

\newcommand{\grs}[1]{\raisebox{-.3cm}{\includegraphics[height=.75cm]{YB#1.pdf}}}
\newcommand{\gra}[1]{\raisebox{-.4cm}{\includegraphics[height=1cm]{SG#1.pdf}}}
\newcommand{\graa}[1]{\raisebox{-.6cm}{\includegraphics[height=1.5cm]{SG#1.pdf}}}
\newcommand{\grb}[1]{\raisebox{-.8cm}{\includegraphics[height=2cm]{SG#1.pdf}}}

\begin{document} 

\newtheorem{theorem}{Theorem}[section]
\newtheorem{lemma}[theorem]{Lemma}

\theoremstyle{definition}
\newtheorem{definition}[theorem]{Definition}
\newtheorem{example}[theorem]{Example}
\newtheorem{xca}[theorem]{Exercise}
\newtheorem{notation}[theorem]{Notation~}
\newtheorem{proposition}[theorem]{Proposition~}
\newtheorem{corollary}[theorem]{Corollary~}

\theoremstyle{remark}
\newtheorem{remark}[theorem]{Remark}

\newtheorem*{ac}{Acknowledgement}

\newpage
\title{Singly generated planar algebras of small dimension, Part III}
\author{Dietmar Bisch, Vaughan F. R. Jones, Zhengwei Liu}
\date{\today}

\begin{abstract}
The first two authors classified subfactor planar algebra generated by a non-trivial 2-box subject to the condition that the dimension of 3-boxes is at most 12 in Part I; 13 in Part II of this series.
They are the group planar algebra for $\mathbb{Z}_3$, the Fuss-Catalan planar algebra ; and the group/subgroup planar algebra for $\mathbb{Z}_2\subset \mathbb{Z}_5\rtimes \mathbb{Z}_2$.
In the present paper, we extend the classification to 14 dimensional 3-boxes. They are all BMW.
Precisely it contains a depth 3 one from quantum $SO(3)$, and a one-parameter family from quantum $Sp(4)$.
\end{abstract}

\maketitle

\section{Introduction}
Interest in subfactors began with \cite{Jon83} where the indices of subfactors of type II$_1$ were shown to be in the set $$\{4\cos^2(\frac{\pi}{n}), n=3,4,\cdots \}\cup [4,\infty].$$
This suggested the index as a complexity measure for subfactors so the simplest subfactors would be those of index less than 4 and then those of index between 4 and 5. Indeed subfactors of index at most 4 were classified \cite{Ocn88,GHJ,Pop94,Izu91}. This approach has been extremely successful in work of Haagerup \cite{Haa94} and others \cite{AsaHaa,Bis98,Izu91,SunVij,BMPS}. Recently the classification has been extended up to index 5, and even beyond. See \cite{ind50,ind51,ind52,ind53,ind54}.

Such classifications would not be possible without the reduction of the subfactor problem to an essentially combinatorial one.
The invariant classifying subfactors is known as the ``standard invariant''  and was axiomatized as
Ocneanu's paragroups \cite{Ocn88} and Popa's $\lambda$-lattices \cite{Pop95}. Later on Jones gave an axiomatization entirely in
terms of planar diagrams in \cite{JonPA}.  The structure is called a planar algebra.  The so-called ``principal graph" is common to these axiomatizations and is to be
thought of as the graph of tensoring irreducible bimodules by $M$ as an $N-M$-bimodule. To say that a subfactor is of finite depth is
to say that its principal graph is finite.
A deep theorem of Popa  \cite{Pop90} shows that the standard invariant
is a complete invariant of subfactors of finite index and finite depth of the hyperfinite II$_1$ factor.

The planar algebra perspective suggests a completely different measure of the complexity of a subfactor.
Planar algebras have presentations in terms of generators and relations so
it is natural to say that the simplest subfactors are those whose planar algebras are
generated by the fewest elements satisfying the simplest relations. The index may be arbitrarily large.
The simplest planar algebra of all is the one with no generators nor relations (!!), which is a sub planar algebra of
any planar algebra, known as the Temperley-Lieb algebra.

The next most complicated planar algebras after Temperley-Lieb should be those generated
by a single element. See \cite{Wen90,MPSD2n,Pet10,BMPS} for examples.

A planar algebra $\mathscr{S}$ consists of vector spaces $\mathscr{S}_{n,\pm}$ for ${n\in \mathbb{N}\cup \{0\}}$.
An element in $\mathscr{S}_{n,\pm}$ is called an $n$-box. (In other interpretations it is a morphism between bimodules.)
Planar algebras generated by a 1-box were completely analyzed by the second author in \cite{JonPA}.
A classification of planar algebras generated by a single 2-box is surely impossible.
But if we are given restrictions on the dimensions of the $\mathscr{S}_{n,\pm}$ many linear combinations of
planar algebra elements must vanish so that we can write down a lot of relations on the generator.
If we are lucky, dimension restrictions will be enough to limit the possibilities for the entire planar algebra. In practice it seems that
if the relations are powerful enough to calculate the value of a labelled planar diagram with no boundary points, then we can
calculate the entire structure just from these relations.
The BMW planar algebras \cite{BirWen,Mur87} are generated by an element satisfying the Yang-Baxter equation or equivalently the
type III Reidemeister move. The evaluation of labeled diagrams with no boundary points is known as the Kauffman polynomial, see \cite{Kau90}. Motivated by this family, the classification for $\dim(\mathscr{S}_{3,\pm})\leq 15$ appears to be possible (apart from a non-generic situation).

The first two authors classified subfactor planar algebras generated by a 2-box for $\dim(\mathscr{S}_{3,\pm})\leq 12$ in \cite{BisJon97}, and for $\dim(\mathscr{S}_{3,\pm})=13$ in \cite{BisJon02}. They are the group  planar algebra for $\mathbb{Z}_3$, the Fuss-Catalan planar algebra ; and the group/subgroup planar algebra for  $\mathbb{Z}_2\subset \mathbb{Z}_5\rtimes \mathbb{Z}_2$.
In this paper we extend the classification to dimension 14 as follows:
\begin{theorem}
Suppose $\mathscr{S}$ is a subfactor planar algebra generated by a 2-box subject to the condition $\dim(\mathscr{S}_{3,\pm})=14$. Then $\mathscr{S}$ is BMW.
More precisely either it is the depth 3 one from quantum $SO(3)$, or it arises from quantum $Sp(4)$.
\end{theorem}

When $\dim(\mathscr{S}_{3,\pm})\leq 13$, the generator satisfies an $exchange$ $relation$ \cite{Bis94,Lan02}. Then the planar algebra is determined by the structure of 2-boxes, i.e., adjoints, contragredients, products and coproducts, which are easily computable.
When $\dim(\mathscr{S}_{3,\pm})=14$, the Yang-Baxter relation is still good enough to show that the whole planar algebra structure is determined by the 2-box operations, see Lemma \ref{skein}. We may derive all the
parameters of  the structure of 2-boxes by  direct computations similar to those of \cite{BisJon02}. To simplify the computations, we will use results proved by the third author in \cite{Liuex} which prove the existence of a $biprojection$ under certain conditions, see Section \ref{bipro}.

\begin{ac}
The third author would like to thank the first two authors for their direction for this paper, Emily Peters for her help on the study of planar algebras, and Hans Wenzl and Eric Rowell for helpful discussions about BMW algebras.

Dietmar Bisch and Zhengwei Liu were supported by NSF Grant DMS-1001560. Vaughan Jones was supported by NSF Grant DMS-0301173.
All authors were supported by DOD-DARPA Grant HR0011-12-1-0009.
\end{ac}

\section{Background}\label{back}
We refer the reader to \cite{Jon12} for the definition of a subfactor planar algebra.

\subsection{Notation}
Suppose $\mathscr{S}=\{\mathscr{S}_{n,\pm}\}_{n\in \mathbb{N}_0}$ is a subfactor planar algebra. We call an element in $\mathscr{S}_{n,\pm}$ an $n$-box.
We use the following notation,
$\delta$ is the value of a closed circle; $id$ is the identity of $\mathscr{S}_{2,+}$; $e$ is the Jones projection of $\mathscr{S}_{2,+}$;
the (unnormalized) Markov trace on  $\mathscr{S}_n$ is denoted by $tr_n(x)=\gra{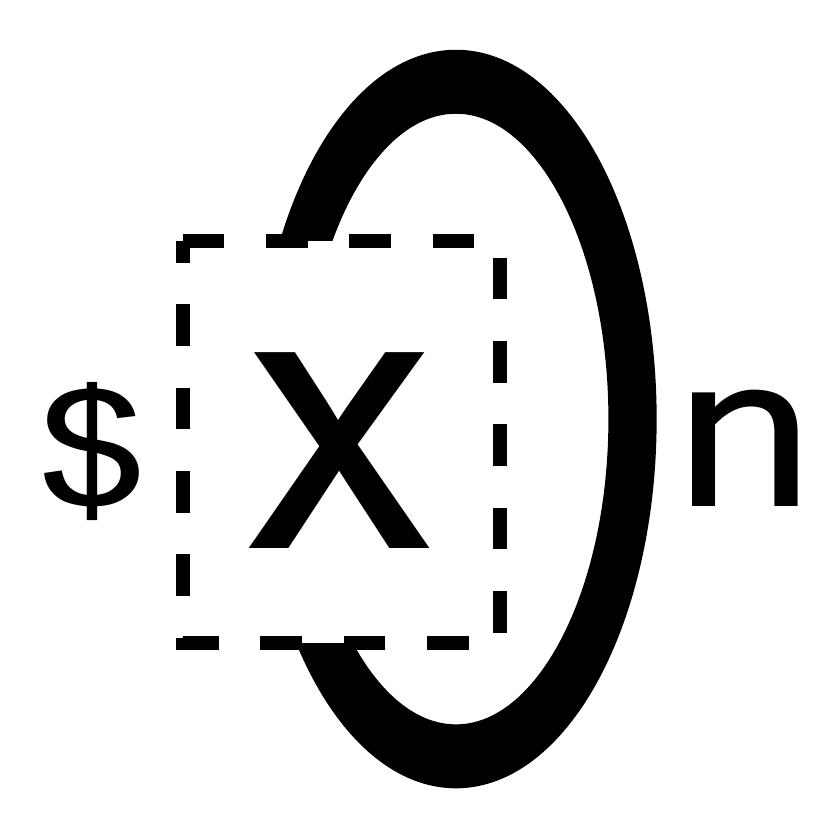}$, $\forall ~x\in\mathscr{S}_n$,
when $n=2$, we write $tr(x)$ for short.
For $a, b\in \mathscr{S}_{2,\pm}$, we define
$\mathcal{F}(a)=\gra{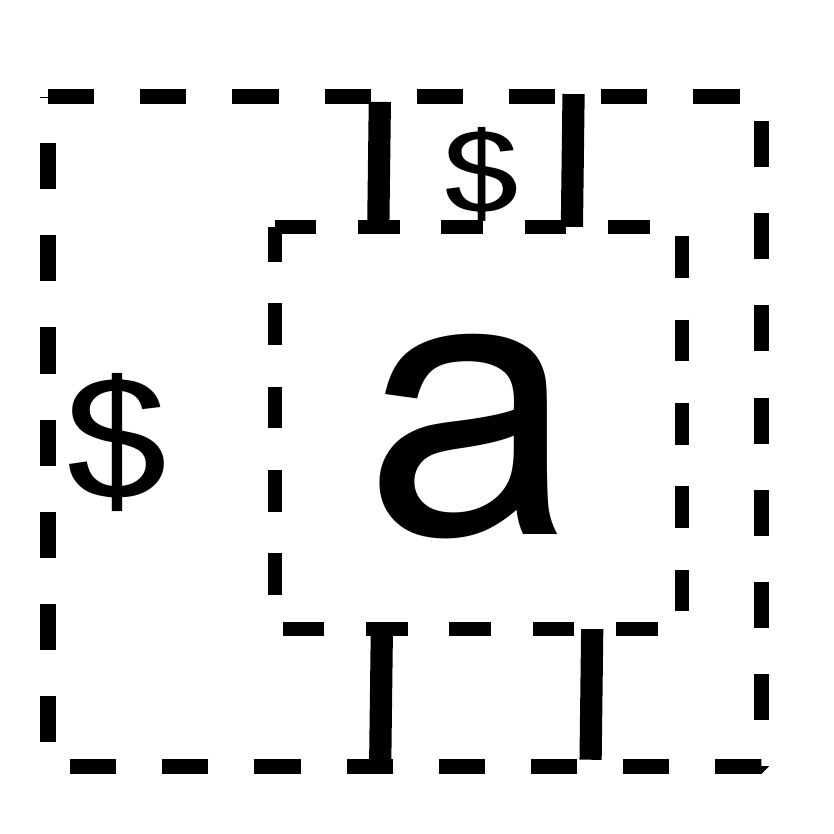}$ to be the 1-click rotation of $a$;
$\overline{a}=\gra{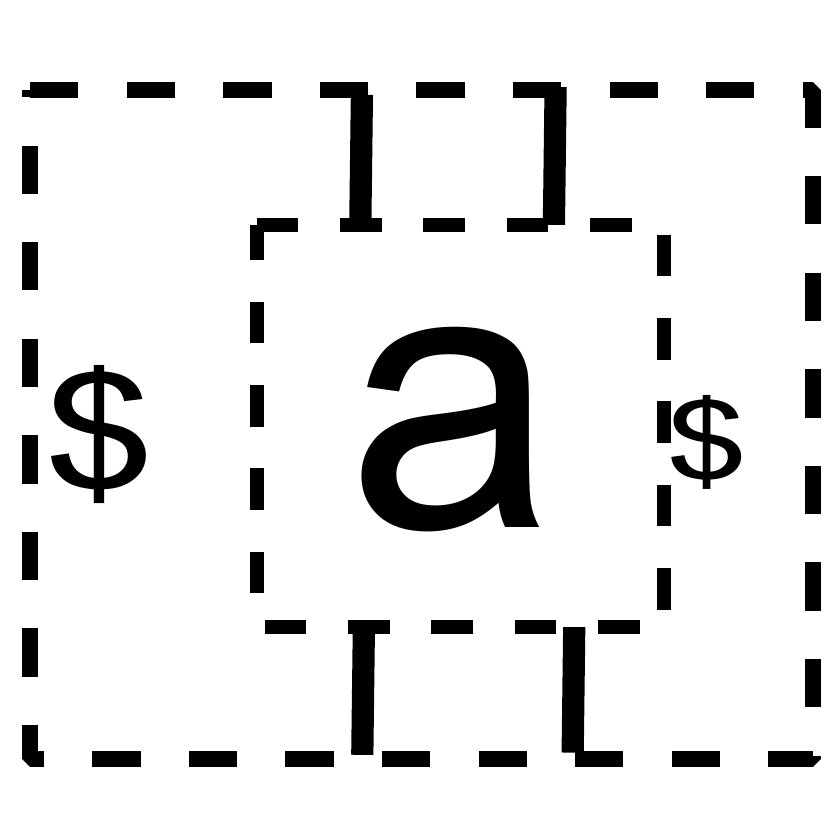}$ to be the contragredient of $a$;
$ab=\graa{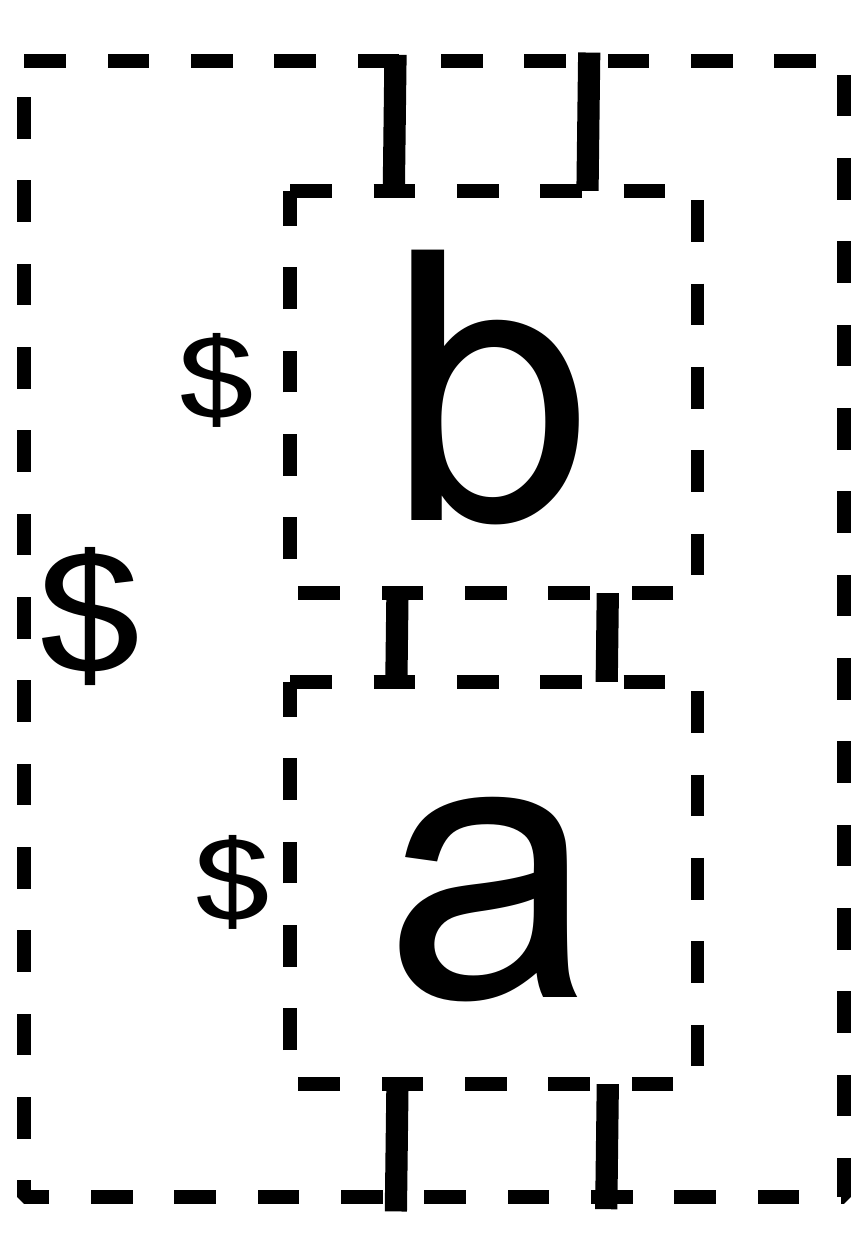}$ to be the product of $a$ and $b$;
$a*b=\gra{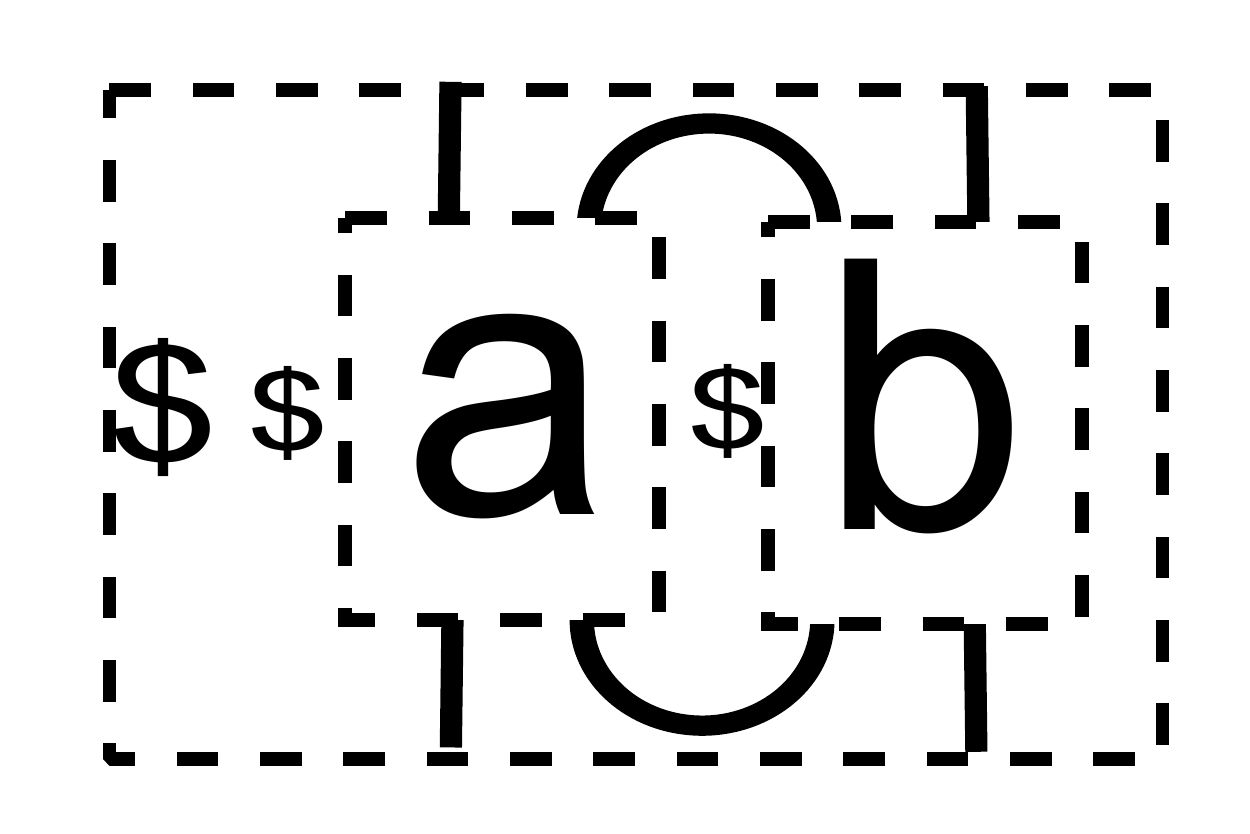}$ to be the coproduct of $a$ and $b$.
It is easy to check that $\mathcal{F}(ab)=\mathcal{F}(a)*\mathcal{F}(b)$.

\begin{definition}
A 2-box $U$ is called bi-invertible, if there is a 2-box $V$ such that $UV=id$, $U*\mathcal{F}^2(V)=\delta e$.
Furthermore it is a bi-unitary if $V=U^*$.
\end{definition}

\begin{definition}
For two self-adjoint operators $x$ and $y$, we say $x$ is weaker (resp. stronger) than $y$ if the support of $x$ (resp. $y$) is a subprojection of the support of $y$ (resp. $x$), written as $x\preceq y$ (resp. $y\succeq x$). If $x\preceq y$ and $y\preceq x$, then they have the same support, written as $x\sim y$.
\end{definition}
For a self-adjoint operator $x$ and a projection $p$, $x\preceq p$ is equivalent to $x=pxp$.

\subsection{Biprojections}\label{bipro}
Biprojections were introduced by the first author while considering the projection onto an intermediate subfactor \cite{Bis94}.
The subfactor planar algebra generated by a biprojection is well understood, named a Fuss-Catalan subfactor planar algebra \cite{BisJonFC}. It has at most 12 dimensional 3-boxes.
Exchange relation planar algebras were introduced by Landau \cite{Lan02} motivated by the exchange relation of a biprojection \cite{Bis94}.
The following results from Section 4 in \cite{Liuex} ensure the existence of a biprojection.

\begin{definition}[Definition 4.2 in \cite{Liuex}]
Suppose $X\in\mathscr{S}_{2,+}$ is a positive operator, and $X=\sum_{i=1}^kC_iP_i$ for some mutually orthogonal minimal projections $P_i\in\mathscr{S}_{2,+}$ and $C_i>0$, for $1\leq i\leq k.$
Let us define the rank of $X$ to be $k$, denoted by $r(X)$.
\end{definition}

\begin{lemma}[Theorem 4.1 and Lemma 4.3 in \cite{Liuex}]\label{thickness}
Suppose $P$ and $Q$ are projections in $\mathscr{S}_{2,+}$. Then $\overline{P}*Q$ is positive and
$$r(\overline{P}*Q)\leq \dim(Q\mathscr{S}_{3,+} P),$$
where $Q\mathscr{S}_3 P=\{QxP| x\in\mathscr{S}_{3,+}\}$.
\end{lemma}

If $P,Q$ are two minimal projections in $\mathscr{S}_{2,+}$, then $P,Q$ correspond to two vertices in the principal graph, and $\dim(Q\mathscr{S}_{3,+}P)$ is the number of length 2 paths between the two vertices.

\begin{theorem}[Theorem 4.10 in \cite{Liuex}]\label{P*P=P}
Suppose $P$ is a projection in $\mathscr{S}_{2,+}$.
If $P*P\preceq P$, equivalently $(P*P)P=P*P$, then $P$ is a biprojection.
\end{theorem}

\begin{definition}[Definition 4.21 in \cite{Liuex}]
Suppose $P$ is a central minimal projection in $\mathscr{S}_2$, such that $tr(P)>1$ and $r(P*Q)=1$ (resp. $r(Q*P)=1$), for any minimal projection $Q$ in $\mathscr{S}_{2,+}$, $Q\neq \overline{P}$. Then we call $P$ a left (resp. right) virtual normalizer. If $P$ is a left and right virtual normalizer, then we call it a virtual normalizer.
\end{definition}


\begin{theorem}[Theorem 4.22 in \cite{Liuex}]\label{virtual normalizer}
Suppose $\mathscr{S}$ is a subfactor planar algebra generated by $\mathscr{S}_2$.
If $\mathscr{S}_2$ contains a left (or right) virtual normalizer, then either $\mathscr{S}$ is Temperley-Lieb or $\mathscr{S}$ is a free product of two non-trivial subfactor planar algebras.
\end{theorem}

\subsection{Skein Theory}\label{skein theory}
Suppose $\mathscr{S}$ is a subfactor planar algebra generated by a 2-box. If $\dim(\mathscr{S}_{3,\pm})=14$, then $\dim(\mathscr{S}_{2,\pm})=3$, since $\dim(\mathscr{S}_{2,\pm})^2\leq \dim(\mathscr{S}_{3,\pm})$.
Let $R$ be the non-Temperley-Lieb 2-box generator viewed as a degree 4 vertex. Considering the diagrams without faces in $\mathscr{S}_{3,+}$, there are 5 Temperley-Lieb diagrams, 6 diagrams with one vertex, 3 diagrams with two vertices as follows,
$$~~\grs{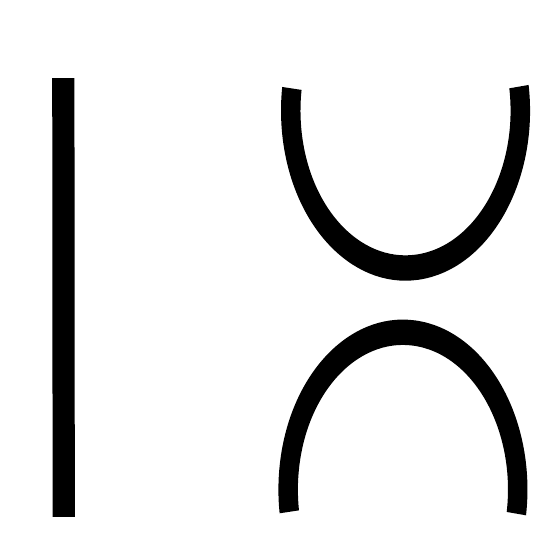}, ~~\grs{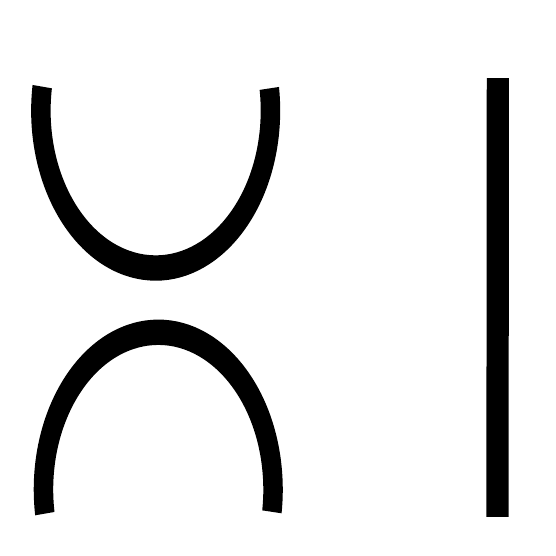}, ~~\grs{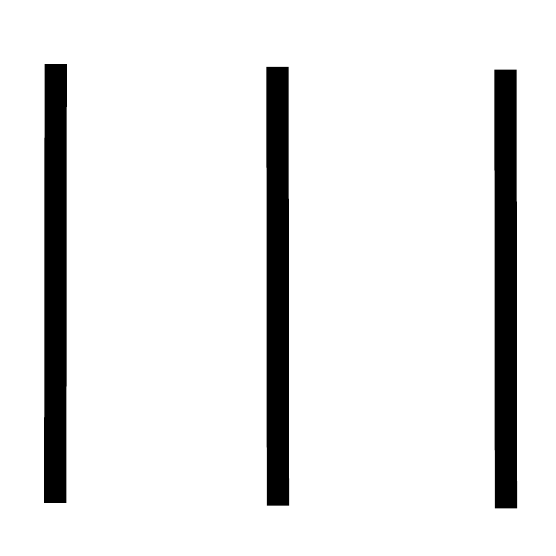}, ~~\grs{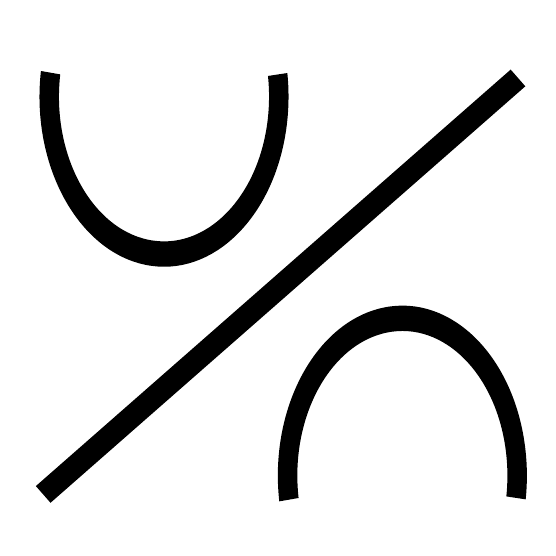}, ~~\grs{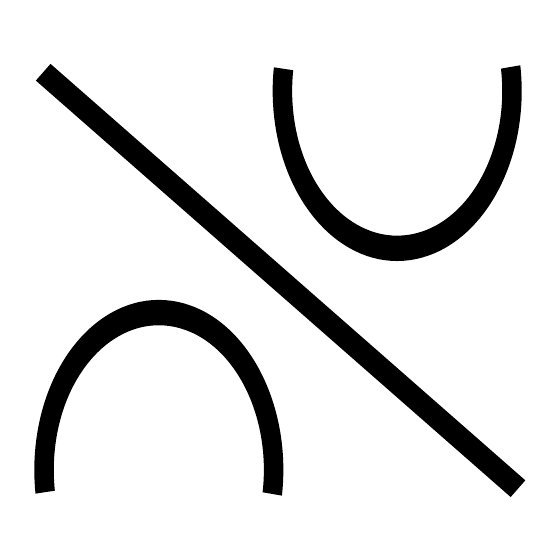}; ~~\grs{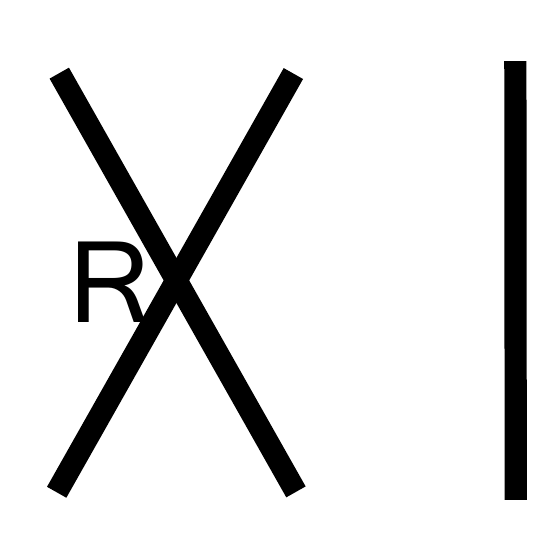}, ~~\grs{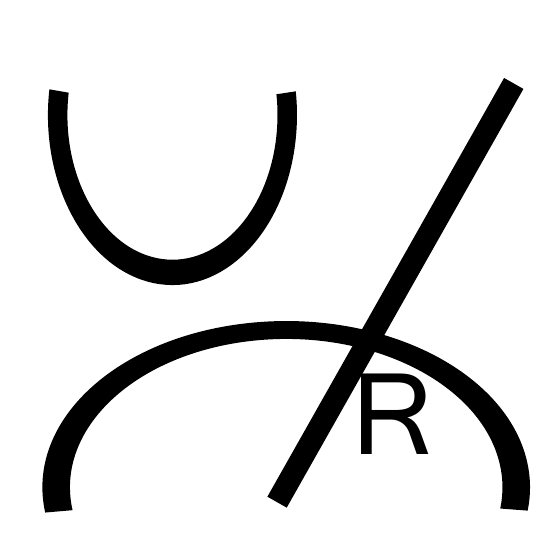}, ~~\grs{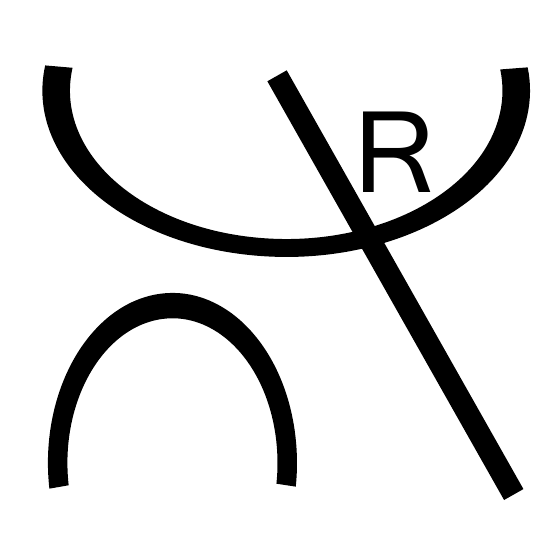}, ~~\grs{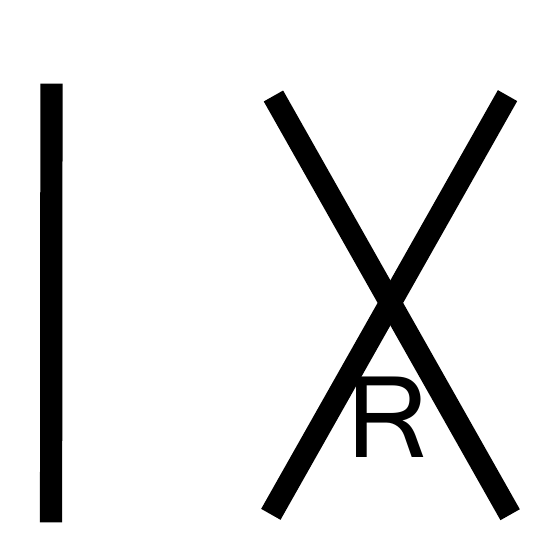}, ~~\grs{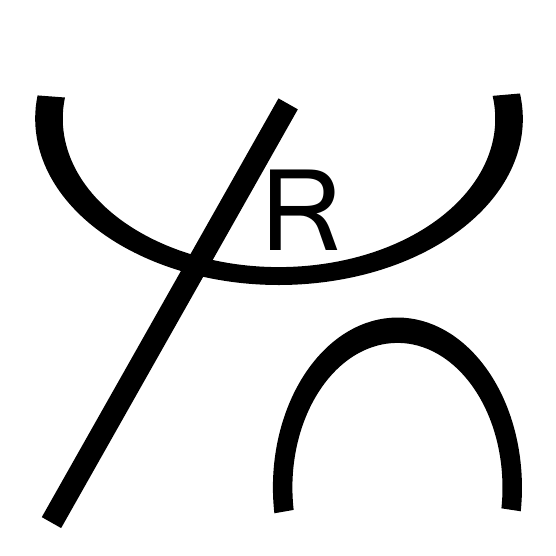}, ~~\grs{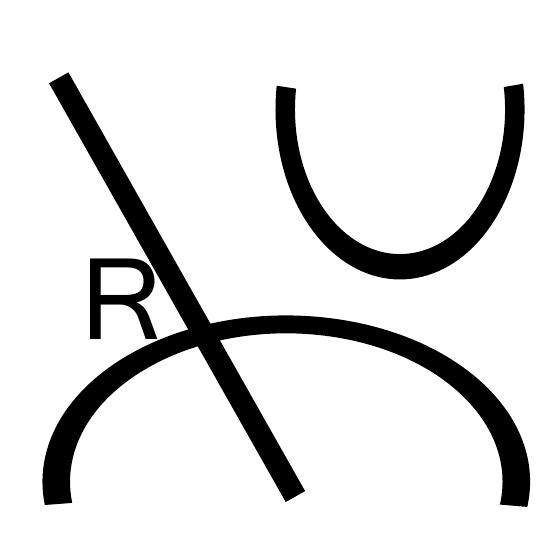}; ~~\grs{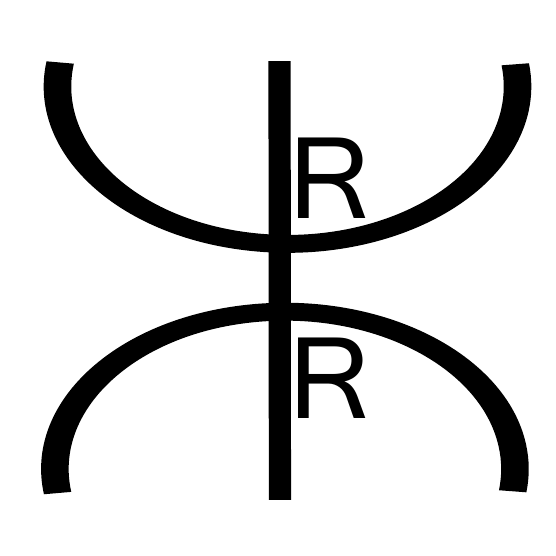}, ~~\grs{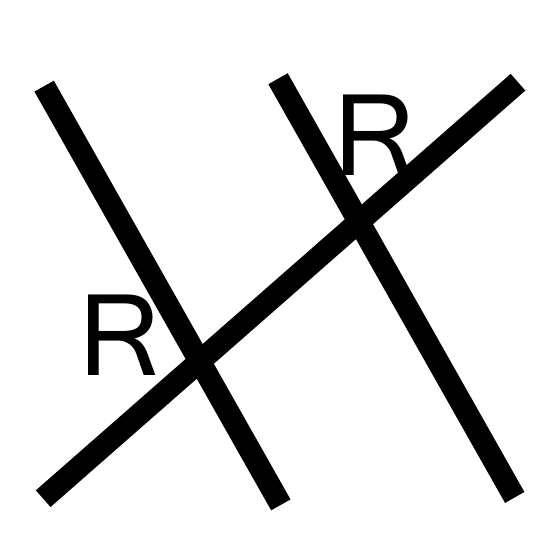}, ~~\grs{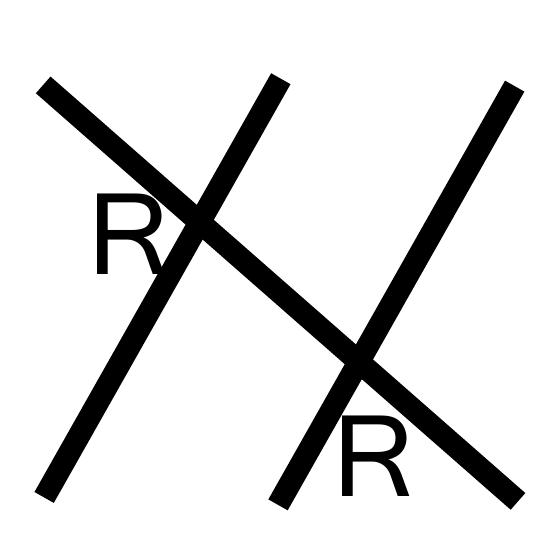},$$
where the position of $R$ indicates the position of $\$$.
Note that the linear span of the 14 diagram does not depend on the choices of the positions of the $\$$'s.

\begin{proposition}
If $\dim(\mathscr{S}_{3,+})=14$, then the 14 diagrams in $\mathscr{S}_{3,+}$ with at most 2 vertices form a basis.
\end{proposition}

\begin{proof}
Let us assume that $\dim(\mathscr{S}_{3,\pm})=14$ and the 14 diagrams are linearly dependent.

If the 11 diagrams with at most one vertex are linearly independent, then one of the diagrams with two vertices is a linear combination of the other 13 diagrams. Up to rotation, we obtain the exchange relation for the generator $R$. Therefore $\mathscr{S}$ is an exchange relation planar algebra and $\dim(\mathscr{S}_{3,\pm})=13$, a contradiction.

If the 11 diagrams with at most one vertex are linearly dependent, then $R$ still satisfies an exchange relation and $\dim(\mathscr{S}_{3,\pm})=13$, a contradiction.

Therefore the 14 diagrams are linearly independent.
\end{proof}

Since the 14 diagrams above form a basis of $\mathscr{S}_{3,+}$, the diagrams $~~\grs{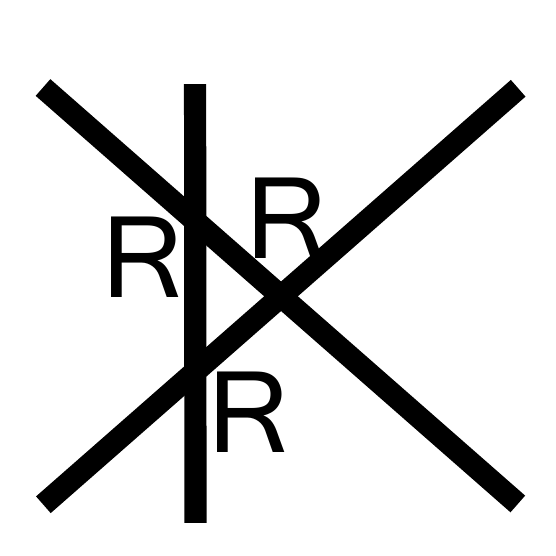}, ~~\grs{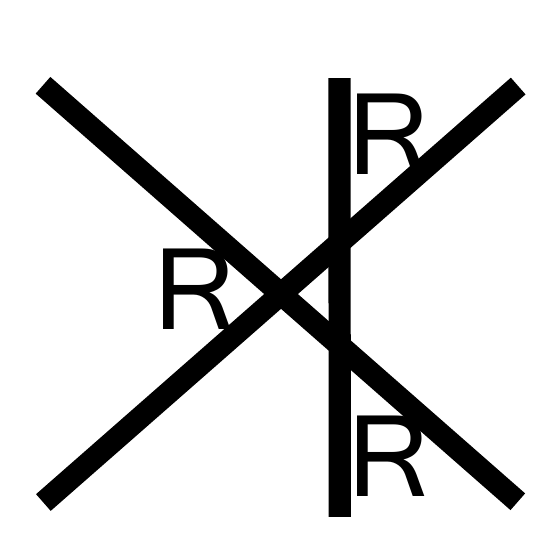}$ reduce to linear combinations of the 14 diagrams. We are going to show that the coefficients only depend on the $structure$ $of$ $2$-$boxes$ defined as follows.

\begin{definition}
The structure of 2-boxes of a subfactor planar algebra consists of the data of adjoints, contragredients, products and coproducts of 2-boxes.
\end{definition}

The following data is also derived from the structure of 2-boxes,
the identity $id$ is identified as the unique unit of 2-boxes under the product; the value of a closed circle $\delta$ is determined by the coproduct of two identities; $\delta e$ is identified as the unique unit of 2-boxes under the coproduct; the trace of a 2-box is determined by its coproduct with the identity $id$. If the planar algebra is irreducible, then capping a 2-box is also determined.

\begin{lemma}\label{reducetriangle}
Suppose that $\dim(\mathscr{S}_{3,+})=14$.
When reducing $~~\grs{15}, ~~\grs{16}$ as linear combinations of the 14 diagrams, the coefficients are determined by the structure of 2-boxes.
\end{lemma}

\begin{proof}
Note that any closed diagram with at most 5 vertices can be reduced to a scalar by the structure of 2-boxes.
So the inner products of the 14 diagrams are determined by the structure of 2-boxes.
Moreover, the inner product of $\grs{15}$ or $\grs{16}$ with one of the 14 diagrams are also determined.
Therefore the coefficients of the 14 diagrams in $\grs{15}$ and $\grs{16}$ are determined by the structure of 2-boxes.
\end{proof}

\begin{lemma}\label{skein}
Suppose $\mathscr{S}$ is a subfactor planar algebra generated by a 2-box subject to the condition $\dim(\mathscr{S}_{3,\pm})=14$, and $\mathscr{A}$ is a subfactor planar algebra generated by a 2-box. If a linear map $\phi: \mathscr{S}_2\rightarrow\mathscr{A}_2$ is surjective and it preserves the structure of 2-boxes, then $\phi$ extends to a planar algebra isomorphism from $\mathscr{S}$ to $\mathscr{A}$.
\end{lemma}

\begin{proof}
We assume that the linear map $\phi: \mathscr{S}_2\rightarrow\mathscr{A}_2$ is surjective and preserves the structure of 2-boxes.
Let us extend $\phi$ to the universal planar algebra generated by the 2-box generator of $\mathscr{S}$.

If $Y$ is a relation of the generator derived from the structure of 2-boxes of $\mathscr{S}$, then $\phi(y)=0$ in $\mathscr{A}$, because $\phi: \mathscr{S}_2\rightarrow\mathscr{A}_2$ preserves the structure of 2-boxes.

If $y$ is a relation which reduces the diagram $a=\grs{15}$ (or $b=\grs{16}$) as a linear combination of the 14 diagrams with at most 2 vertices in $\mathscr{S}$, then $tr(y^*y)=0$.
Note that $tr(\phi(y)^*\phi(y))$ is a linear combination of closed diagrams with at most 5 vertices and $tr(a*a)$ (or $tr(b^*b)$). The evaluation of these closed diagrams in $\mathscr{A}$ only depends on the structure of 2-boxes, so $tr(\phi(y)^*\phi(y))=tr(y^*y)=0$.
Then $\phi(y)=0$ in $\mathscr{A}$ by the positivity of the trace.

By Euler's formula, a closed diagram consisting of degree 4 vertices contains a face with at most 3 edges.
By Lemma \ref{reducetriangle}, the relations reducing a face with at most 3 edges to diagrams without a face are determined by the structure of 2-boxes. By the above arguments, the image of these relations from $\mathscr{S}$ under $\phi$ are 0 in $\mathscr{A}$.
By assumption, $\mathscr{S}$ is generated by a 2-box, so $\phi$ induces a planar algebra $*$-homomorphism from the quotient $\mathscr{S}$ to $\mathscr{A}$.
By positivity of the trace, any planar algebra $*$-homomorphism of subfactor planar algebras is injective.
By assumption, $\phi$ is surjective on 2-boxes, and $\mathscr{A}$ is generated by 2-boxes, so $\phi$ is a planar algebra $*$-isomorphism.
\end{proof}

It is easy to generalize this result to the case of multiple generators.

\subsection{BMW}\label{bmw14}
We cite the conventions in \cite{Wen90}.
The Birman-Murakami-Wenzl (BMW) algebra \cite{BirWen,Mur87} is a two-parameter family of (unshaded, unoriented) planar algebras $\mathscr{C}(r,q)$ generated by a self-contragredient bi-invertible element $U=\gra{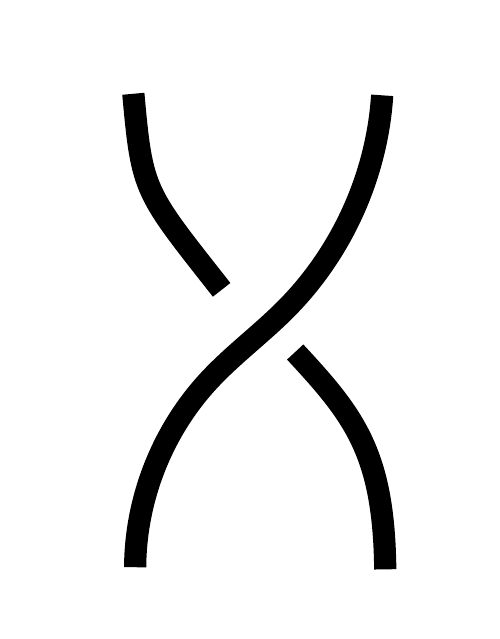}$ as a solution of the Yang-Baxter equation with the following relations
$$\gra{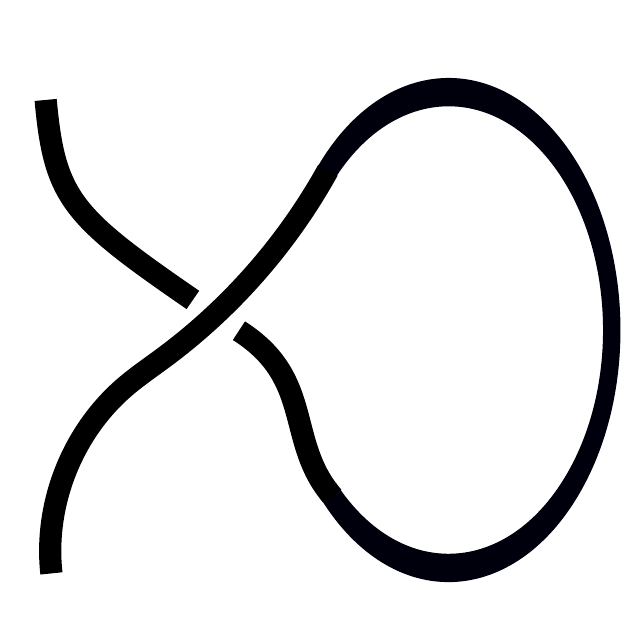}=r~\gra{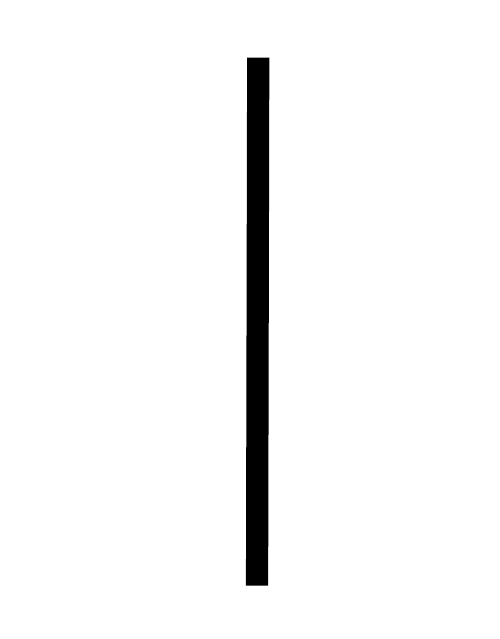};\quad\gra{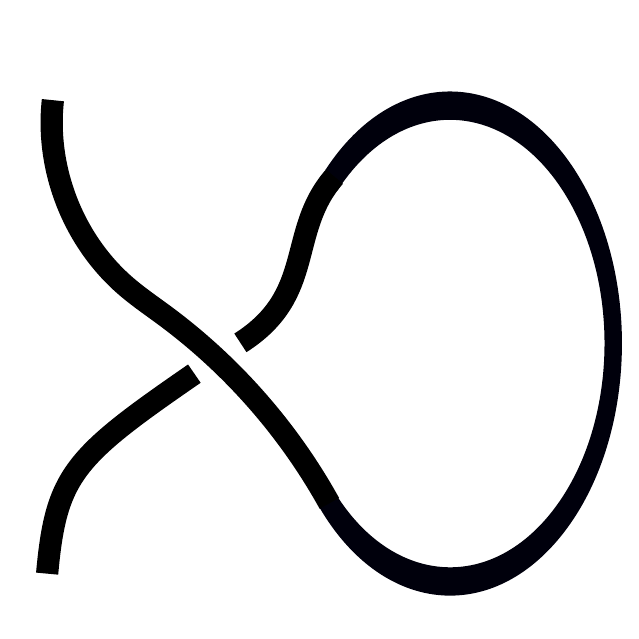}=r^{-1}~\gra{br3};$$
$$\gra{br4}-\gra{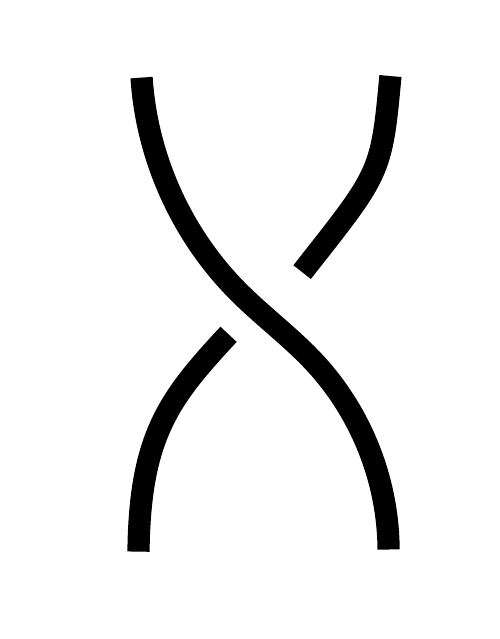}=(q-q^{-1})(\gra{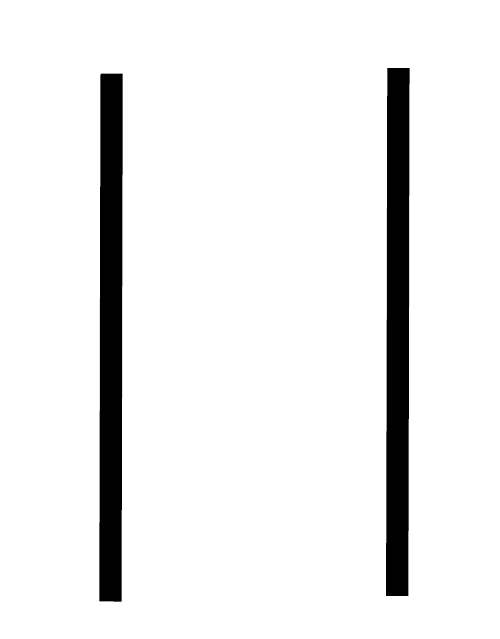}-\gra{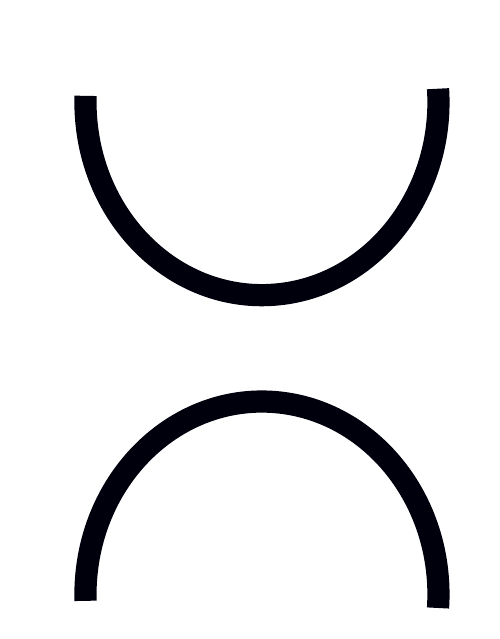}).$$
These relations correspond to Reidemester move I, II, III of a braid $\gra{br4}$ and its quadratic equation.
The value of a closed circle is derived as $\displaystyle \delta'=\frac{r-r^{-1}}{q-q^{-1}}+1$. When $q=r=1$, BMW reduces to the Brauer algebra parameterised by $\delta'$.

\begin{remark}
Note that $\delta'$ could be negative. In this case, we will switch the Jones projection to its negative, then the value of a closed circle becomes $-\delta'$.
\end{remark}

When the ground field is $\mathbb{Q}(q,r)$, rational functions over $q$ and $r$, the minimal idempotents of BMW are labeled by Young diagrams. The trace formula is given by Theorem 5.5 in \cite{Wen90}.
When $q,r$ are fixed complex numbers, and the groud field is $\mathbb{C}$, the idempotents are constructed inductively by skein theory in \cite{BelBla}. This process stops once the trace of a minimal idempotent is 0, see Section 8 in \cite{BelBla}.

\begin{remark}
The function $Q_\lambda$ in Theorem 5.5 in \cite{Wen90} is a multiple of the trace, see page 404 in \cite{Wen90} for the proportion. When $\displaystyle \frac{r-r^{-1}}{q-q^{-1}}+1>0$, $Q_\lambda$ is our unnormalised Markov trace $tr_n$ on n-boxes. When $\displaystyle \frac{r-r^{-1}}{q-q^{-1}}+1<0$, $(-1)^nQ_\lambda$ is our unnormalised Markov trace $tr_n$ on n-boxes.
\end{remark}

\begin{notation}
We use the convention that $[n]$ is the Young diagram with $1$ row and $n$ columns; $[1^n]$ is the Young diagram with $n$ rows and 1 column.
\end{notation}

The generator $U$ could be expressed as a linear combination of minimal idempotents in 2-box space,
\begin{align*}
U&=r^{-1}e+q P_1 -q^{-1} P_2,
\end{align*}
where $P_1, P_2$ are minimal idempotents labeled by the Young diagrams $[2]$ and $[1^2]$ respectively.
By Theorem 5.5 in \cite{Wen90}, we have
\begin{align*}
tr(P_1)&=\frac{(rq-r^{-1}q^{-1}+q^{2}-q^{-2})(r-r^{-1})}{(q^{2}-q^{-2})(q-q^{-1})};\\
tr(P_2)&=\frac{(rq^{-1}-r^{-1}q+q^{2}-q^{-2})(r-r^{-1})}{(q^{2}-q^{-2})(q-q^{-1})}.
\end{align*}

\begin{lemma}\label{unique braid}
The element $U=r^{-1}e+q P_1 -q^{-1} P_2$ is the unique solution of Reidemester move I and the quadratic equation in BMW over $\mathbb{Q}(q,r)$.
\end{lemma}

\begin{proof}
Suppose $U=c_0 e+ c_1 P_1 +c_2 P_2$ is a solution of Reidemester move I and the quadratic equation. Adding a cap at the bottom of $U$, we have $c_0=r^{-1}$. By the quadratic equation, we have $c_i-c_i^{-1}=q-q^{-1}$, for $i=1,2$ So $c_i=q$ or $c_i=-q^{-1}$. Adding a cap on the right of $U$, we have $\delta' r-r^{-1}=c_1 tr(P_1)+c_2 tr(P_2)$. Among the four choices of $(c_1,c_2)$, only $U=r^{-1}e+q P_1 -q^{-1} P_2$ satisfies this equation.
\end{proof}

The traces of $P_1$ and $P_2$ are determined by the parameters $q$ and $r$ of BMW. Now let us solve $q$ and $r$ by the two traces.
\begin{lemma}\label{unique q r}
Given real numbers $\delta',a,b$, if there are complex numbers $q,r$, such that
\begin{align*}
\delta'&=\frac{r-r^{-1}}{q-q^{-1}}+1;\\
a&=\frac{(rq-r^{-1}q^{-1}+q^{2}-q^{-2})(r-r^{-1})}{(q^{2}-q^{-2})(q-q^{-1})};\\
b&=\frac{(rq^{-1}-r^{-1}q+q^{2}-q^{-2})(r-r^{-1})}{(q^{2}-q^{-2})(q-q^{-1})},
\end{align*}
then
\begin{align*}
q^2+q^{-2}&=\frac{2(\delta'-1)^4-4(\delta'-1)^2+2(b-a)^2}{(\delta'-1)^4-(b-a)^2};\\
r&=\frac{(\delta'-1)^2(q-q^{-1})+(a-b)(q-q^{-1})}{2(\delta'-1)}.
\end{align*}
In this case either $|q|=|r|=1$ or $q,r$ are reals.
Furthermore if $\Re{q}\geq0$, $\Im{q}\geq 0$, then $q,r$ are uniquely determined by $\delta',a,b$.
\end{lemma}

\begin{align*}
q^2+q^{-2}&=\frac{2(1+\delta)^4-4(1+\delta)^2+2(b-a)^2}{(1+\delta)^4-(b-a)^2};\\
r&=\frac{-(1+\delta)^2(q-q^{-1})+(b-a)(q-q^{-1})}{2(1+\delta)}.
\end{align*}

\begin{proof}
By our assumptions, we have
\begin{align}
\frac{r-r^{-1}}{q-q^{-1}}&=\delta'-1; \label{equ1}\\
\frac{rq-r^{-1}q^{-1}}{q^{2}-q^{-2}}&=\frac{a}{\delta'-1}-1; \label{equ2}\\
\frac{rq^{-1}-r^{-1}q}{q^{2}-q^{-2}}&=\frac{b}{\delta'-1}-1. \label{equ3}
\end{align}
Equation $(\ref{equ1})$ and $(\ref{equ2})-(\ref{equ3})$ imply
\begin{align*}
r-r^{-1}&=(\delta'-1)(q-q^{-1});\\
r+r^{-1}&=\frac{(a-b)(q+q^{-1})}{\delta'-1}.
\end{align*}
Then
\begin{align*}
r&=\frac{(\delta'-1)^2(q-q^{-1})+(a-b)(q-q^{-1})}{2(\delta'-1)}.
\end{align*}
Note that $(r-r^{-1})^2+4=(r+r^{-1})^2$, so
\begin{align*}
(\delta'-1)(q-q^{-1}))^2+4&=(\frac{(a-b)(q+q^{-1})}{\delta'-1})^2.
\end{align*}
Then
\begin{align*}
q^2+q^{-2}&=\frac{2(1+\delta)^4-4(1+\delta)^2+2(b-a)^2}{(1+\delta)^4-(b-a)^2}.
\end{align*}
The right side is real, so either $|q|=1$ or $q,r$ are reals.
If $|q|=1$, then $r-r^{-1}=(\delta'-1)(q-q^{-1})$ is imaginary. So $|r|=1$.
If $q$ is real, then $r-r^{-1}=(\delta'-1)(q-q^{-1})$ is real. So $r$ is real.
Furthermore if $\Re{q}\geq0$, $\Im{q}\geq 0$, then $q$ is determined by the real number $q^2+q^{-2}$.
So $q,r$ are uniquely determined by $\delta',a,b$.
\end{proof}

\begin{lemma}
The four BMW planar algebras $\mathscr{C}(r,q)$, $\mathscr{C}(-r,-q)$, $\mathscr{C}(r^{-1},q^{-1})$, $\mathscr{C}(-r^{-1},q)$ are isomorphic.
\end{lemma}

\begin{proof}
Take $U=r^{-1}e+q P_1 -q^{-1} P_2$ to be the generator of $\mathscr{C}(r,q)$. Then
$U_1=-U$, $U_2=U^{-1}$, $U_3=-re-q^{-1}P_1+q P_2$ are generators of $\mathscr{C}(-r,-q)$, $\mathscr{C}(r^{-1},q^{-1})$, $\mathscr{C}(-r^{-1},q)$ respectively.
\end{proof}

We are interested in the case that $|q|=|r|=1$ or $q,r\in\mathbb{R}$ which admits a natural involution. Up to an isomorphisms, in the former case we assume that $\Re{q}\geq0,\Im{q}\geq0,\Re{r}\geq0$; in the latter case we assume that $q\geq1, r\geq0$.
To obtain a subfactor planar algebra, its Markov trace will be positive semidefinite, and the quotient of the planar algebra by the kernel of the partition function is a subfactor planar algebra.

\begin{notation}
The proper quotient of BMW is denoted by $\pi(\mathscr{C}(r,q))$, see \cite{Wen90,JonPA} for details.
\end{notation}

When $|q|=1$, these subfactor planar algebras are classified in Table 1 in \cite{Wen90}.
It is worth mentioning that the second-to-last row in Table 1 should be excluded.
This is clarified in \cite{Row05}, and Theorem 3.8 in \cite{Row08}.
When $q\geq1$, the argument is similar to Corollary 5.6 in \cite{Wen90}. If $r\neq q^n$, then the structure trace is non-degenerate on BMW.

If $r^{-1}\leq {q}^{n-1}<r<q^{n}$, then by Theorem 5.5 in \cite{Wen90} the trace of the minimal idempotent corresponding to $[1^{n+2}]$ is negative, where $[1^{n+2}]$ is the Young diagram with $n+3$ rows and 1 column. So it is not a subfactor planar algebra.

If $r^{-1}\geq {q}^{-2n-1}>r>q^{-2n-3}$, then the trace of the minimal idempotent corresponding to $[n+2]$ is negative, where $[n+2]$ is the Young diagram with 1 row and $n+2$ columns. So it is not a subfactor planar algebra.

If $r=q^{-1}$, then $\delta=0$. So it is not a subfactor planar algebra.

When $r=q^{n-1}$ or $r=q^{-2n-1}$, for $n\geq 1$, these subfactor planar algebras are constructed from quantum $SO(n)$ and $Sp(2n)$ in \cite{Saw95}.

\begin{remark}
In general $\displaystyle q=e^{\frac{\pi i}{l}}$, $l=mk+h$, where $h$ is the dual Coxeter number, $k$ is the level and $m=1$,$2$ or $3$ depending on whether the Dynkin diagram is simply-laced or not.
\end{remark}

In general the dimension of 3-boxes of BMW is 15. For our purpose, we hope the dimension to be 14.
That means the trace of the minimal projection corresponding to the Young diagram $[3]$ or $[1^3]$ is 0.
Then we obtain two cases.

Case 1: $r=q^2$ and $\displaystyle q=e^{\frac{2\pi i}{7}}$. It arises from quantum $SO(3)$. Its principal graph is $\gra{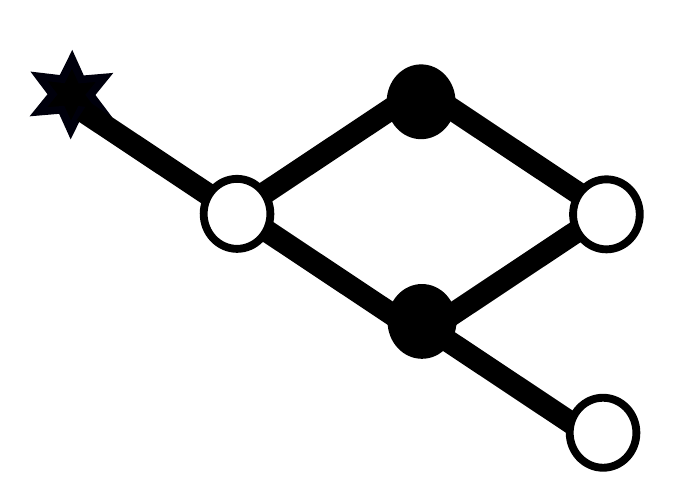}.$

Case 2: $r=q^{-5}$ and $\displaystyle q=e^{\frac{\pi i}{l}}$, for $l$ even, $l\geq 12$, or $q\geq 1$. They arise from quantum $Sp(4)$. In this case the value of a closed circle $\delta'$ is negative. To obtain a positive value, we need to switch the Jones projection to its negative. Then we have
$$\gra{br1}=-r~\gra{br3};\quad\gra{br2}=-r^{-1}~\gra{br3};$$
$$\gra{br4}-\gra{br5}=(q-q^{-1})(\gra{br6}+\gra{br7});$$
and $\displaystyle \delta=-\delta'=-\frac{r-r^{-1}}{q-q^{-1}}-1=q^4+q^2+q^{-2}+q^{-4}$.

\begin{remark}
To derive the above formulas while switching the Jones projection, one way is considering the formulas as multiplication of generators; another way is considering the winding number.
\end{remark}



\section{Classification}\label{classification}
Suppose $\mathscr{S}$ is a subfactor planar algebra generated by a 2-box and $\dim(\mathscr{S}_{3,\pm})=14$, then $\dim(\mathscr{S}_{2,\pm})=3$, and $\mathscr{S}_{2,+}$, $\mathscr{S}_{2,-}$ are abelian. Suppose $P_1,P_2\in\mathscr{S}_{2,+}, Q_1,Q_2\in\mathscr{S}_{2,-}$ are distinct minimal projections orthogonal to the Jones projections, such that $tr(P_1)\leq tr(P_2)$ and $tr(Q_1)\leq tr(Q_2)$. Take $a=tr(P_1)$, $b=tr(P_2)$, and $a\leq b$, then $a+b=\delta^{2}-1$.
First we will show the unique possibility of the principal graph up to depth 3. Consequently the generator is self-contragredient, and $tr(P_i)=tr(Q_i)$, for $i=1,2$. Then we compute the coproduct of $\mathscr{S}_{2,+}$, which is determined by $a,b$. Furthermore $a,b$ are determined by $\delta$, by computing the chirality. When the chirality is 1, $\mathscr{S}$ is depth 3. When the chirality is $-1$, we construct the bi-invertible generator in $\mathscr{S}_{2,+}$ with its relation, which implies $\mathscr{S}$ is BMW.

\begin{lemma}\label{prin}
The (dual) principal graph of $\mathscr{S}$ up to depth 3 is
$$\graa{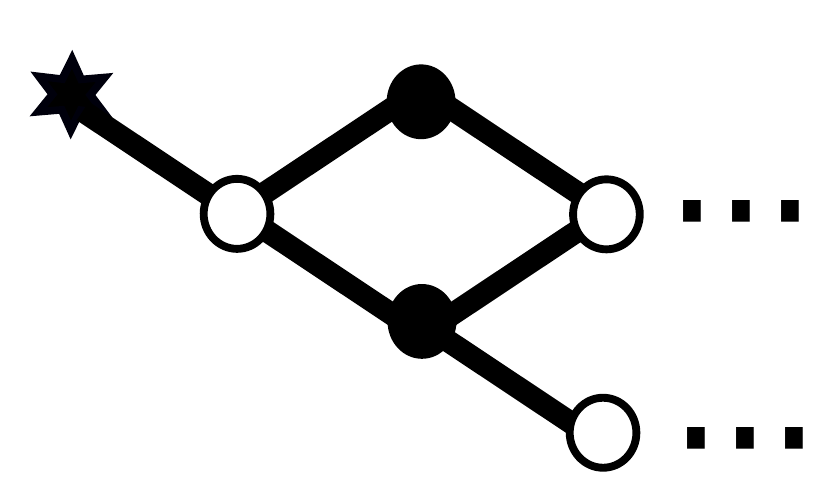}.$$
\end{lemma}

\begin{proof}
Since $\dim(\mathscr{S}_{2,\pm})=3$, there are two depth 2 vertices in the principal graph.
If there is no common depth 3 vertex adjacent to the two depth 2 vertices, then $\dim(P_2\mathscr{S}_{3,+}P_1)=1$.
By Lemma \ref{thickness}, we have $r(\overline{P_1}*P_2)=1$. If $tr(P_2)=1$, then $tr(P_1)=1$ and $\delta^2=3$, a contradiction. If $tr(P_2)>1$, then $P_2$ is a right virtual normalizer. By Theorem \ref{virtual normalizer}, $\mathscr{S}$ is Fuss-Catalan, a contradiction.
So there is a common depth 3 vertex adjacent to the two depth 2 vertices. By counting the dimension of $\mathscr{S}_{3,+}$, there is one more depth 3 vertex in the principal graph. So the principal graph up to depth 3 is as depicted above.
\end{proof}

\begin{corollary}\label{cor}
The projections $P_i, Q_i$ are self-contragredient, and $tr(P_i)=tr(Q_i)$, for $i=1,2$.
\end{corollary}

\begin{proof}
Considering the duality of odd vertices between the principal graph and the dual principal graph, we have $\delta tr(Q_2)=\delta tr(P_2)>\delta tr(P_1)=\delta tr(Q_1)$.
So $P_i, Q_i$ are self-contragredient, and $tr(Q_i)=tr(P_i)$, for $i=1,2$.
\end{proof}

\begin{lemma}\label{copr}
We have the following formulas for coproducts:
\begin{align*}
P_1*P_1&=\frac{a}{\delta}e + \frac{a^2-a}{\delta b}P_2;\\
P_1*P_2&=\frac{a-1}{\delta}P_1 + \frac{ab-a^2+a}{\delta b}P_2;\\
P_2*P_2&=\frac{b}{\delta}e+\frac{b-a+1}{\delta}P_1 + \frac{b^2-b-ab+a^2-a}{\delta b}P_2.\\
\end{align*}

Similar formulas hold for $Q_1$ and $Q_2$.
\end{lemma}

\begin{proof}
By Lemma \ref{prin} and Lemma \ref{thickness}, we have $r(P_1*P_1)\leq 2$. If $P_1*P_1\preceq e+P_1$, then by Theorem \ref{P*P=P}, $e+P_1$ is a biprojection, and $\mathscr{S}$ is Fuss-Catalan, a contradiction.
Note that $e\preceq P_1*P_1$, so $P_1*P_1\sim e+P_2$.

The coefficient of $e$ in $P_1*P_1$ is $tr((P_1*P_1)e)$, and $\displaystyle tr((P_1*P_1)e)=\frac{a}{\delta}$ by isotopy.
Note that $\displaystyle tr(P_1*P_1)=\frac{a^2}{\delta}$, by computing the trace, we have
\begin{align*}
P_1*P_1&=\frac{a}{\delta}e + \frac{a^2-a}{\delta b}P_2.
\end{align*}

By isotopy, we have
\begin{align*}
tr((P_1*P_2)e)&=0, \\
tr((P_1*P_2)P_1)&=tr((P_1*P_1)P_2)=\frac{a^2-a}{\delta}.
\end{align*}
So the coefficients of $e$ and $P_1$ in $P_1*P_2$ are $0$ and $\displaystyle \frac{a-1}{\delta}$ respectively.
By computing the trace, we have
\begin{align*}
P_1*P_2=\frac{a-1}{\delta}P_1 + \frac{ab-a^2+a}{\delta b}P_2.
\end{align*}

Similarly we have
\begin{align*}
tr((P_2*P_2)e)&=\frac{b}{\delta},\\
tr((P_2*P_2)P_1)&=tr((P_1*P_2)P_2)=\frac{ab-a^2+a}{\delta},
\end{align*}
and
\begin{align*}
P_2*P_2&=\frac{b}{\delta}e+\frac{b-a+1}{\delta}P_1 + \frac{b^2-b-ab+a^2-a}{\delta b}P_2.
\end{align*}
\end{proof}

\begin{corollary}\label{unique hat}
There is a unique subfactor planar algebra whose principal graph is $\gra{pringra}$.
\end{corollary}

\begin{proof}
The existence follows from case 1 in Section \ref{bmw14}.
If a subfactor has this principal graph, then $\delta,a,b$ are fixed, and the structure of 2-boxes are derived from Corollary \ref{cor} and Lemma \ref{copr}.
Its uniqueness follows from Lemma \ref{skein}.
\end{proof}

This result was also proved by S. Morrison and E. Peters in \cite{MorPet12}.

In Section \ref{skein theory}, we showed that the two triangles in $\mathscr{S}_{3,+}$ have to reduce to a linear combination of diagrams with at most 2 vertices. The following Lemma tells how to reduce the two triangles.

\begin{lemma}\label{ppp0}
\begin{align*}
\grb{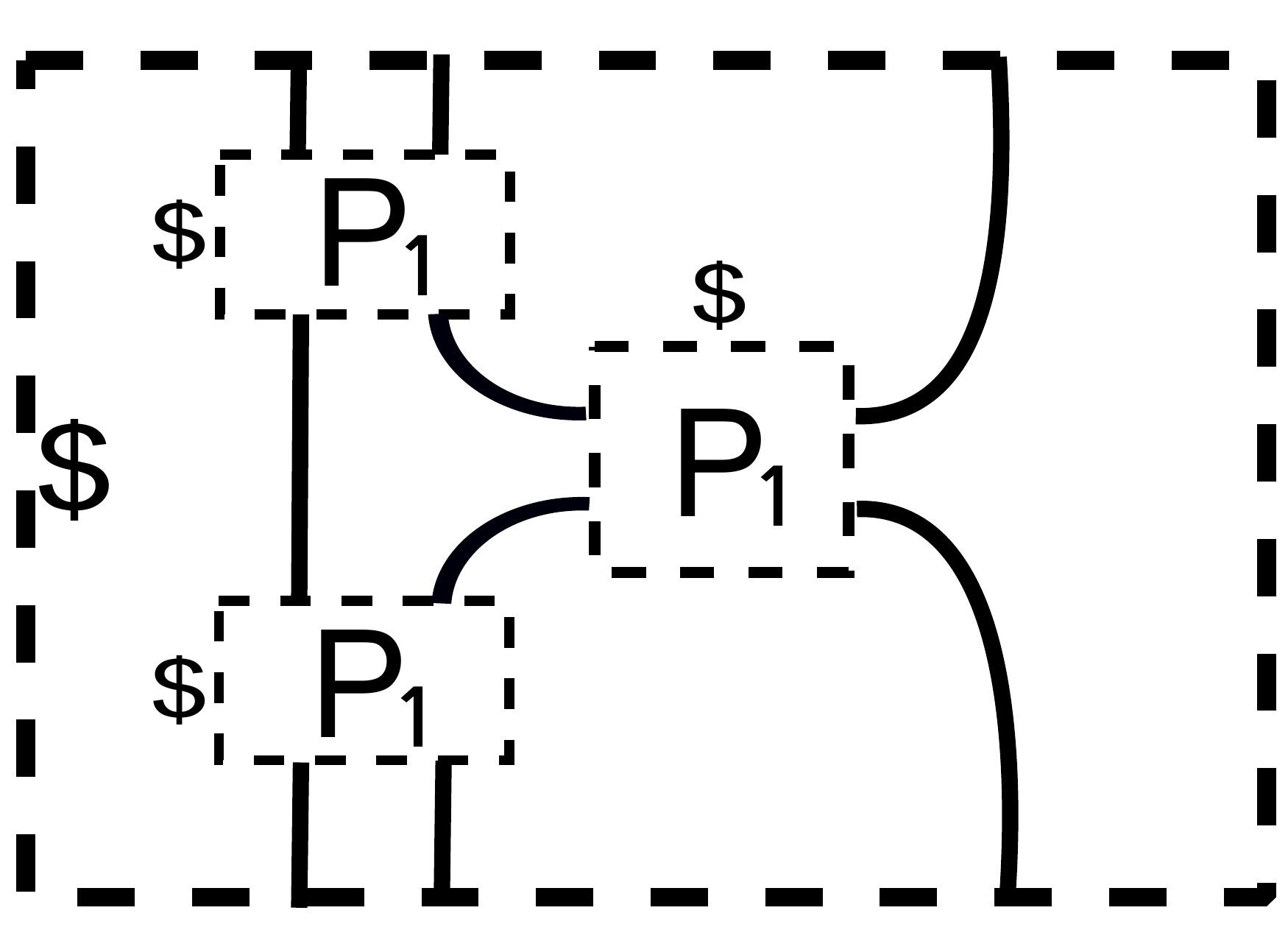}=0.
\end{align*}
A similar formula holds for $Q_1$.
\end{lemma}

\begin{proof}
Let $x$ be the diagram in the statement. Note that $\overline{P_1}=P_1$ and $P_1$ is a projection, by isotopy we have $tr_3(x^*x)=tr((P_1*P_1)P_1)$. By Lemma \ref{copr} $tr((P_1*P_1)P_1)=0$, so $tr_3(x^*x)=0$. By positivity of the trace, we have $x=0$.
\end{proof}

\begin{theorem}\label{uniqueness}
Suppose $\mathscr{S}$ is a subfactor planar algebra generated by a 2-box subject to the condition $\dim(\mathscr{S}_{3,\pm})=14$. Then $\mathscr{S}$ is uniquely determined by $\delta$. If $\mathscr{S}$ is not the depth 3 one, then
$\mathcal{F}(bP_1-aP_2)=-(bQ_1-aQ_2)$, and
\begin{align*}
\frac{b}{a}=\frac{\delta-3+(\delta-1)\sqrt{4\delta+9}}{2\delta}.
\end{align*}
\end{theorem}

\begin{proof}
By Corollary \ref{unique hat}, we may assume that $\mathscr{S}$ is not the depth 3 subfactor planar algebra.
Note that $bP_1-aP_2$, $bQ_1-aQ_2$ are uncappable, self-adjoint, self-contragredient, and they have the same 2-norm $\sqrt{ab^2-a^2b}$, so
\begin{align*}
\mathcal{F}(bP_1-aP_2)&=\pm(bQ_1-aQ_2).
\end{align*}
Moreover,
\begin{align*}
&\mathcal{F}((bP_1-aP_2)*(bP_1-aP_2))\\
=&(bQ_1-aQ_2)^2\\
=&ab(Q_1+Q_2)+(b-a)(bQ_1-aQ_2).
\end{align*}

If $\mathcal{F}(bP_1-aP_2)=bQ_1-aQ_2$, then
$$(bP_1-aP_2)*(bP_1-aP_2)=ab(\delta e-\delta^{-1}id)+(b-a)(bP_1-aP_2).$$
Applying Lemma \ref{copr} and comparing the coefficients of $(bP_1-aP_2)$, we have
\begin{align*}
3ab-a^2(a+b-1)&=\delta b(b-a).
\end{align*}
Replacing $a+b$ by $\delta^2-1$, $\displaystyle \frac{b}{a}$ by $y$, we have
\begin{align}\label{equ 1}
\delta y^2-(\delta+3)y+\delta^2-2&=0.
\end{align}
So
\begin{align*}
0&\leq (\delta+3)^2-4\delta(\delta^2-2)=(9-4\delta)(\delta+1)^2.
\end{align*}
Then $\delta\leq \frac{9}{4}$.
There are two possible ways to add one depth 4 vertex and one edge in the principal graph. Their graph norms are about 5.18 and 5.44. While adding more vertices and edges, the graph norm will increase.
But $\delta^2\leq(\frac{9}{4})^2=5.0625<5.18$. So $\mathscr{S}$ is the unique depth 3 one.

If $\mathcal{F}(bP_1-aP_2)=-(bQ_1-aQ_2)$, then
$$(bP_1-aP_2)*(bP_1-aP_2)=ab(\delta e-\delta^{-1}id)-(b-a)(bP_1-aP_2).$$
Similarly we have
$$\delta y^2-(\delta-3)y-(\delta^2-2)=0.$$
Note that $\displaystyle y=\frac{b}{a}>0$, so
$$y=\frac{\delta-3+(\delta-1)\sqrt{4\delta+9}}{2\delta}.$$
Recall that $a+b=\delta^2-1$, so
\begin{align*}
a&=\frac{1}{y+1}(\delta^2-1),\\
b&=\frac{y}{y+1}(\delta^2-1).
\end{align*}
By Lemma (\ref{skein}) and (\ref{copr}), $\mathscr{S}$ is uniquely determined by $\delta$.
\end{proof}

\begin{remark}
If $\mathscr{S}$ is depth 3, then $b=\delta, a=\frac{\delta}{\delta-1},$ and $\delta$ is the largest root of $\delta^3-2\delta^2-\delta+1=0$. It follows from the above proof that $\mathcal{F}(bP_1-aP_2)=bQ_1-aQ_2$.
\end{remark}

If $\mathscr{S}$ is not the depth 3 one, then $\mathscr{S}$ is uniquely determined by $\delta$. When
\begin{align*}
\delta=q^4+q^2+q^{-2}+q^{-4},  ~q&=e^{\frac{\pi i}{l}},
\end{align*}
for $l$ even, $l\geq 12$, or $q\geq 1$, we know that such a subfactor planar algebra exists, namely BMW from quantum $Sp(4)$ \cite{Wen90}.
We cannot yet determine if the remaining one-parameter family are also BMW planar algebras from quantum $Sp(4)$, since we used positivity to derive the classification.
We expect to identify this one-parameter family as BMW from quantum $Sp(4)$. The idea is to find the generator of BMW in $\mathscr{S}_{2,+}$ satisfying the relations, Reidemeister moves I, II, III and the quadratic equation. Then this family is BMW. The generator of BMW in Lemma \ref{unique braid} is parameterised by $q$ and $r$. However, $\mathscr{S}_{2,+}$ is parameterised by $\delta$. We need to solve $q$ and $r$ in terms of $\delta$.

\begin{theorem}\label{BMW}
Suppose $\mathscr{S}$ is a subfactor planar algebra generated by a 2-box subject to the condition $\dim(\mathscr{S}_{3,\pm})=14$, then $\mathscr{S}$ is BMW.
More precisely, it is either the depth 3 one from quantum $SO(3)$, or it arises from quantum $Sp(4)$.
\end{theorem}

The two cases are listed at the end of Section \ref{bmw14}.

\begin{proof}
If $\mathscr{S}$ is depth 3, then by Lemma \ref{prin} and Corollary \ref{unique hat}, it is the depth 3 one from quantum $SO(3)$.
If $\mathscr{S}$ is not depth 3, then it is uniquely determined by $\delta$ and $\mathcal{F}(bP_1-aP_2)=-(bQ_1-aQ_2)$ by Theorem \ref{uniqueness}.
So $\mathscr{S}$ is $\pi(\mathscr{C}(q^{-5},q))$, when $\delta^2$ is the index of $\pi(\mathscr{C}(q^{-5},q))$, for $\displaystyle q=e^{\frac{\pi i}{l}}$, $l$ even, $l\geq 12$, or $q\geq 1$,

Recall that $a+b=\delta^{2}-1$, when $\delta>1$ we have $a>0,b>0$ and
$$\frac{2(\delta+1)^4-4(\delta+1)^2+2(b-a)^2}{(\delta+1)^4-(b-a)^2}>0$$
Take
\begin{align*}
\delta'&=-\delta;\\
r&=\tilde{q}^{-5};\\
\tilde{\delta}&=\frac{r-r^{-1}}{\tilde{q}-\tilde{q}^{-1}}+1;\\
\tilde{a}&=\frac{(r\tilde{q}-r^{-1}\tilde{q}^{-1}+\tilde{q}^{2}-\tilde{q}^{-2})(r-r^{-1})}{(\tilde{q}^{2}-\tilde{q}^{-2})(\tilde{q}-\tilde{q}^{-1})}; \\
\tilde{b}&=\frac{(r\tilde{q}^{-1}-r^{-1}\tilde{q}+\tilde{q}^{2}-\tilde{q}^{-2})(r-r^{-1})}{(\tilde{q}^{2}-\tilde{q}^{-2})(\tilde{q}-\tilde{q}^{-1})}.
\end{align*}
where $\tilde{q}\in\mathbb{C}$ is the solution of
$$\tilde{q}^2+\tilde{q}^{-2}=\frac{2(\delta'-1)^4-4(\delta'-1)^2+2(b-a)^2}{(\delta'-1)^4-(b-a)^2},$$
such that $\Re{\tilde{q}}\geq0$, $\Im{\tilde{q}}\geq0$.

When $\delta^2$ is the index of $\pi(\mathscr{C}(q^{-5},q))$, for $q\geq1$, we have $\tilde{q}=q, \tilde{\delta}=\delta'=-\delta, \tilde{a}=a, \tilde{b}=b$ by Lemma \ref{unique q r}.
Set $x=\sqrt{4\delta+9}$, then $\delta, a, b$ are rational functions of $x$, so $\delta', \tilde{q}^2+\tilde{q}^{-2}$ are rational functions of $x$.
Note that
\begin{align*}
\tilde{\delta}&=-\tilde{q}^4-\tilde{q}^{2}-\tilde{q}^{-2}-\tilde{q}^{-4};\\
\tilde{a}&=(-\tilde{q}^{2}-\tilde{q}^{-2}+1)(\tilde{\delta}-1);\\
\tilde{b}&=(-\tilde{q}^{4}-\tilde{q}^{-4})(\tilde{\delta}-1)
\end{align*}
are polynomials of $\tilde{q}^2+\tilde{q}^{-2}$, so they are rational functions of $x$.
Note that $\tilde{\delta},\tilde{a},\tilde{b}$ are identical to $\delta,a,b$ respectively at infinitely many values of $x$, so they are the same for any $x=\sqrt{4\delta+9}$, $\delta>1$. Therefore $\delta, a ,b$ are rational functions of $\tilde{q}$.

Take $U=r^{-1}(-e)+\tilde{q} P_1 -\tilde{q}^{-1} P_2$, then $U^{-1}=r(-e)+\tilde{q}^{-1} P_1 -\tilde{q} P_2$. When $\tilde{q}=q\geq1$,
by Lemma \ref{unique braid} $U$ is the generator of $\pi(\mathscr{C}(q^{-5},q))$ satisfying Reidemester move I, II, III and the quadratic equation, where the Reidemester move III is given by the Yang-Baxter equation
$$\grb{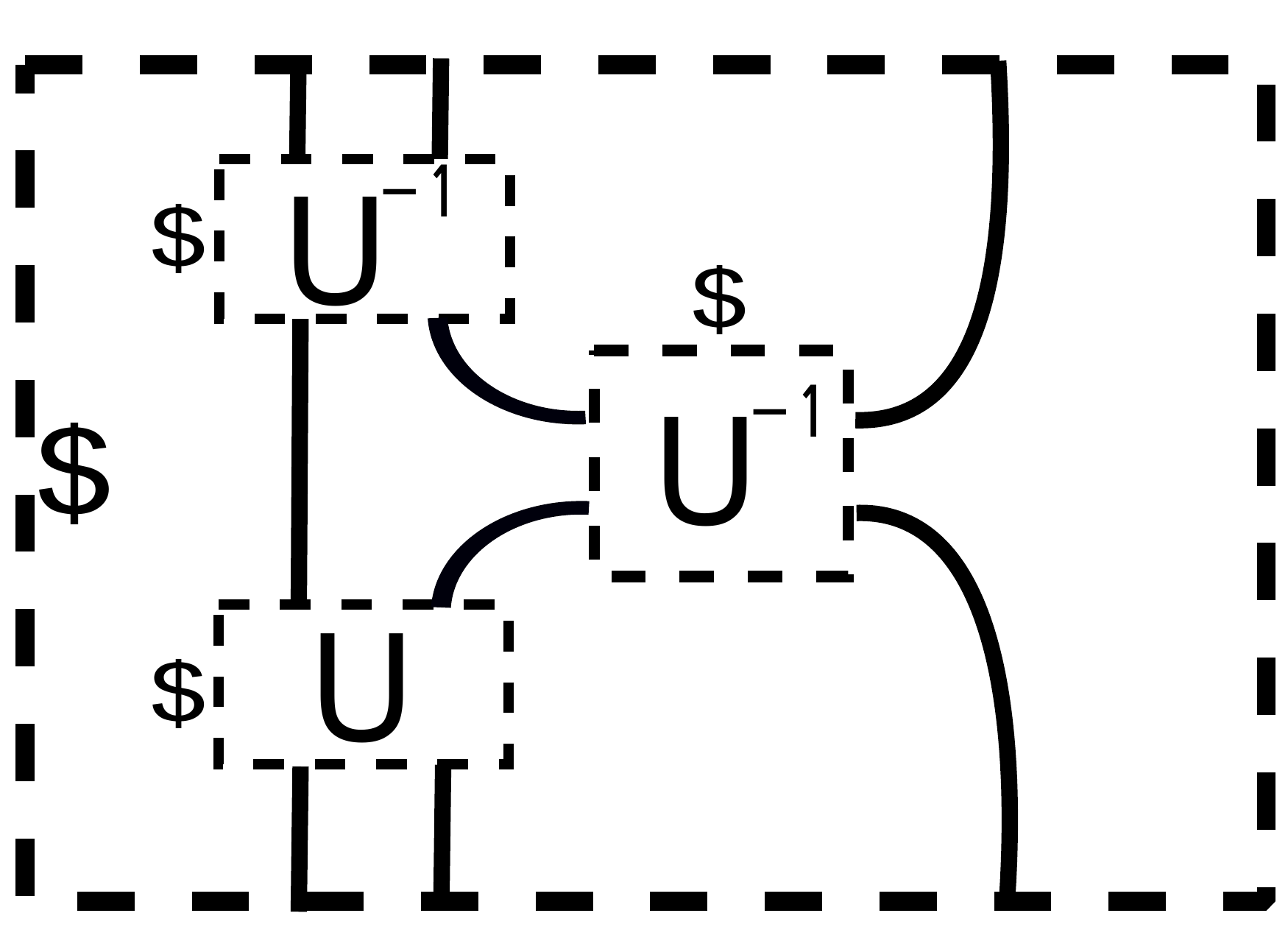}=\grb{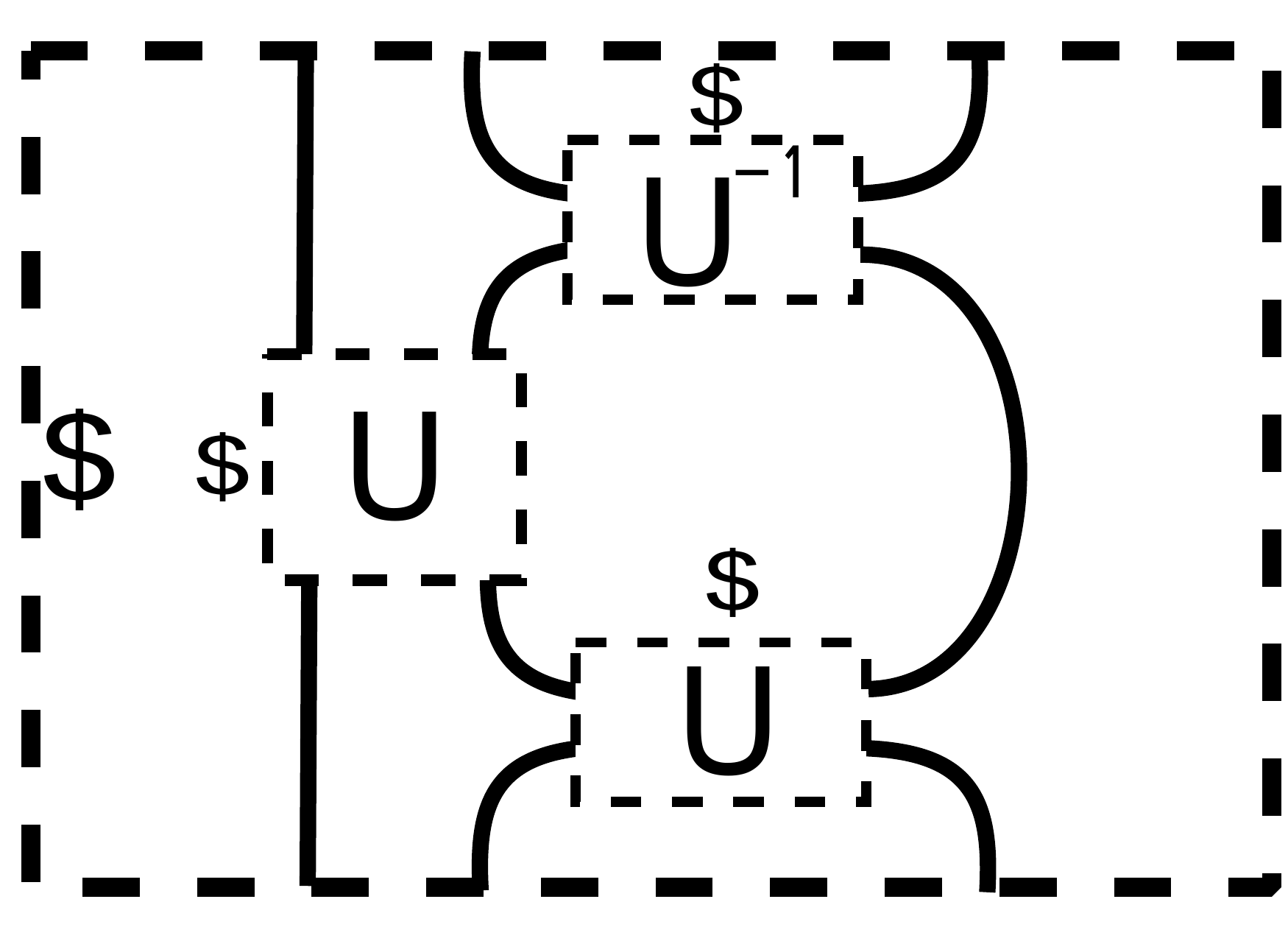}.$$

For general $\tilde{q}$, by Lemma \ref{ppp0} the difference of the above two diagrams is a linear combination of 14 diagrams labeled by at most two $U$'s, and the 14 coefficients are rational functions of $\delta,a,b$. So they are rational functions of $\tilde{q}$. These rational functions are zeros at $\tilde{q}\geq1$, so they are zeros for any $\tilde{q}$. That means $U$ satisfies the Yang-Baxter equation. A similar argument works for Reidemester move I, II and the quadratic equation of $U$. Therefore $U$ is the generator of BMW and $\mathscr{S}$ is BMW.
\end{proof}

\begin{remark}
When $\delta\leq 4$, we have $q^2+q^{-2}\leq 2$, and $U$ is a bi-unitary.
\end{remark}

We give a second proof of Theorem \ref{BMW} without applying Lemma \ref{ppp0}. When $\mathscr{S}$ is not depth 3, we are going to show that the structure of 2-boxes of $\mathscr{S}$ is the same as BMW. Then by Lemma \ref{skein}, $\mathscr{S}$ is BMW.

\begin{proof}
When $\mathscr{S}$ is not depth 3, then it is uniquely determined by $\delta$ and $\mathcal{F}(bP_1-aP_2)=-(bQ_1-aQ_2)$ by Theorem \ref{uniqueness}.
Note that $a+b=\delta^{2}-1$, then
$$\frac{2(\delta+1)^4-4(\delta+1)^2+2(b-a)^2}{(\delta+1)^4-(b-a)^2}=\frac{(\delta b+a)^2+(\delta a+b)^2-(a+b)^2}{(\delta b+a)(\delta a+b)}.$$
Take $q$ to be the solution of
$$q^2+q^{-2}=\frac{(\delta b+a)^2+(\delta a+b)^2-(a+b)^2}{(\delta b+a)(\delta a+b)},$$
such that $\Re(q)\geq \cos(\frac{\pi}{4})$ and $\Im(q)\geq0$. Take
\begin{align*}
z_1&=\frac{(\delta b+a)q-(\delta a+b)q^{-1}}{a+b}; \\
z_2&=\frac{(\delta a+b)q-(\delta b+a)q^{-1}}{a+b}; \\
\mu_1&=\frac{\delta z_2-z_1}{a+b};\\
\mu_2&=\frac{\delta z_1-z_2}{a+b}; \\
\mu_3&=\frac{q+q^{-1}}{a+b};\\
U&=\mu_1id+\mu_2 \delta e+\mu_3(bP_1-aP_2).
\end{align*}
We will see $U$ is the bi-invertible generator of BMW.

Recall that $a+b=\delta^2-1$, so
\begin{align*}
\delta z_2-z_1&=aq-bq^{-1}; & \delta z_1-z_2&=bq-aq^{-1};\\
\mu_1+\mu_2\delta&=z_1; & \mu_1\delta+\mu_2&=z_2; \\
\mu_1+\mu_3b&=q; & \mu_1-\mu_3 a&=-q^{-1}; \\
\mu_2+\mu_3a&=q; & \mu_2-\mu_3 b&=-q^{-1}.
\end{align*}
Observe that
\begin{align*}
z_1+z_2&=\frac{(\delta+1)(a+b)(q-q^{-1})}{a+b}=(\delta+1)(q-q^{-1});\\
z_1z_2&=\frac{(\delta b+a)(\delta a+b)(q^2+q^{-2})-(\delta b+a)^2-(\delta a+b)^2}{(a+b)^2}=-1.
\end{align*}
So
\begin{align*}
z_i-z_i^{-1}&=(q-q^{-1})(\delta+1), \text{ for } i=1,2.
\end{align*}
Then
\begin{align*}
U&=(\mu_1+\mu_2 \delta) e+(\mu_1+\mu_3b)P_1+(\mu_1-\mu_3a)P_2\\
&=z_1 e+ q P_1 -q^{-1}P_2.
\end{align*}
So
\begin{align*}
U^{-1}&=z_1^{-1} e+ q^{-1}P_1 -qP_2,\\
U-U^{-1}&=(q-q^{-1})(id+\delta e).
\end{align*}
By Theorem \ref{uniqueness}, we have $\mathcal{F}(bP_1-aP_2)=-(bQ_1-aQ_2)$. So
\begin{align*}
U&=(\delta\mu_1+\mu_2 ) \delta^{-1}id+(\mu_2-\mu_3b)\mathcal{F}(Q_1)+(\mu_2+\mu_3a)\mathcal{F}(Q_2)\\
&=z_2\delta^{-1}id-q^{-1}\mathcal{F}(Q_1)+q\mathcal{F}(Q_2).
\end{align*}
Take
$$V=z_2^{-1}\delta^{-1}id-q\mathcal{F}(Q_1)+q^{-1}\mathcal{F}(Q_2),$$
then $U*V=\delta e$, and
$$U-V=(q-q^{-1})(id+\delta e).$$
Therefore $V=U^{-1}$, and $U$ is a bi-invertible.
Then the structure of 2-boxes of $\mathscr{S}$ is the same as that of BMW. So $\mathscr{S}$ is BMW by Lemma \ref{skein} . The ones with 14 dimensional 3-boxes are listed in Section \ref{bmw14}.
\end{proof}

  \bibliography{bibliography}
  \bibliographystyle{amsalpha}

\end{document}